\documentclass[12pt]{article}

\usepackage[margin=1in]{geometry}
\usepackage{amsmath,amssymb,amsthm}
\usepackage{enumitem}
    \setlist{itemsep=0mm}
\usepackage{pifont}
\usepackage{csquotes}

\usepackage{tikz}
\usetikzlibrary{positioning,fit}
\usepackage{caption}
\usepackage{graphicx}

\usepackage[citestyle=authoryear,bibstyle=authoryear]{biblatex}
\newtheorem{theorem}{Theorem}
\newtheorem{lemma}{Lemma}
\newtheorem{corollary}{Corollary}
\newtheorem{definition}{Definition}

\newtheorem{remark}{Remark}

\DeclareSymbolFont{symbolsC}{U}{txsyc}{m}{n}
\SetSymbolFont{symbolsC}{bold}{U}{txsyc}{bx}{n}
\DeclareMathSymbol{\coloneqq}{\mathrel}{symbolsC}{"42}
\DeclareMathSymbol{\Coloneqq}{\mathrel}{symbolsC}{"46}

\newcommand{\tup}[1]{\langle #1 \rangle}

\newcommand{\would}{\mathbin{>}}
\newcommand{\nec}{\mathop{\Box}}
\newcommand{\pos}{\mathop{\Diamond}}
\newcommand{\then}{\mathbin{\rightarrow}}
\renewcommand{\iff}{\mathbin{\leftrightarrow}}

\newcommand{\class}[1]{\mathcal{#1}}
\newcommand{\WS}{\class{WS}}
\renewcommand{\S}{\class{S}}
\newcommand{\Ax}{\text{AX}}

\title{Stalnaker's Logical Problem of Conditionals is Unsolvable}
\author{Alexander W.\ Kocurek, James Walsh, Yale Weiss}

\date{}

\begin{document}

\maketitle

\begin{abstract}
    The logical problem of conditionals, as conceived by \cite{Stalnaker1968theory}, amounts to axiomatizing a particular semantics for conditionals which utilizes selection functions that take propositions (i.e., sets of worlds) as arguments. 
    While the sentential form of this semantics is recursively axiomatizable, we prove that its enrichment with first-order quantifiers is not---that is, we show that Stalnaker's logical problem of conditionals is unsolvable in the language with first-order quantifiers. 
    We demonstrate this by showing how to interpret arithmetic in the logic.
    In the conclusion, we discuss the implications of this result for the study of conditional logic.

    \bigskip\noindent\textbf{Keywords.} Arithmetic; Conditionals; Incompleteness; Propositions; Quantifiers; Selection functions.
\end{abstract}

\section{The Logical Problem of Conditionals}
\label{Section:The Logical Problem of Conditionals}

One of the central problems in the literature on conditionals is to characterize their logic. 
\cite{Stalnaker1968theory} nicely articulates this task in his seminal paper on conditionals:
\begin{displayquote}[\cite{Stalnaker1968theory}, p.~98]
    My principal concern will be with what has been called the \emph{logical problem of conditionals}\dots 
    This is the task of describing the formal properties of the \emph{conditional function}: a function, usually represented in English by the words `if\dots then', taking ordered pairs of propositions into propositions.
\end{displayquote}
Stalnaker goes on to ``defend a solution'' to the logical problem of conditionals. 
In doing so, he presents a semantics for conditionals, the \emph{selection function semantics}, which has become arguably \emph{the} standard semantic framework adopted by theorists working on conditionals. 
Thus, Stalnaker takes the logical problem of conditionals to ultimately reduce to the problem of characterizing the logical properties of selection functions. 

Later, \cite{StalnakerThomason1970} presented a completeness theorem for a selection function semantics in a first-order setting. 
One might reasonably conclude from this that this constitutes at least \emph{a} solution to the logical problem of conditionals (even if it is debatable whether this solution is correct). 
However, \cite{KocurekWalshWeissA} have recently disputed this. 
They observe that the logic Stalnaker and Thomason identified, which Stalnaker and Thomason call $\mathsf{CQ}$ (or which \cite{KocurekWalshWeissA} call $\mathsf{QST}$), does not actually address the logical problem of conditionals as understood by \cite{Stalnaker1968theory}. 
For as Stalnaker makes plain in the quote above, the problem is not just to analyze the logic of a certain linguistic construction: it is to analyze the logic of the \emph{conditional function}, that is, the operation on propositions (understood as sets of worlds) encoded by the English `if\dots then' construction. 
Stalnaker has repeatedly emphasized the philosophical importance of viewing selection functions as functions on propositions; only functions on propositions can play the theoretical role that Stalnaker and others have assigned to conditionals (for some discussion, see \cite[Section~1]{KocurekWalshWeissA}). Yet what \cite{StalnakerThomason1970} proved is that $\mathsf{CQ}$ is sound and complete for a version of the selection function semantics in which selection functions take \emph{formulas}, not propositions, as arguments. 
While interesting in its own right, this result does not establish that $\mathsf{CQ}$ is the logic of the selection function semantics that \cite{Stalnaker1968theory} defended. 
Indeed, \cite{KocurekWalshWeissA} prove that not only is $\mathsf{CQ}$ not the logic of Stalnaker's semantics, but moreover, $\mathsf{CQ}$ is \emph{frame incomplete}, meaning it is not the logic of \emph{any} class of first-order selection function frames. 
In other words, assuming the conditional function is accurately analyzed in terms of the selection function semantics, $\mathsf{CQ}$ cannot constitute a solution to the logical problem of conditionals, regardless of the constraints one places on selection functions. 

This result reopens the logical problem of conditionals for Stalnakerians. 
Now, Stalnakerians have a partial solution to this problem: the conditional logic $\mathsf{C2}$ is sound and (weakly) complete for the \emph{sentential} version of Stalnaker's semantics, even when selection functions take propositions as arguments.\footnote{On the weak completeness and failure of strong completeness of \textsf{C2}, see \cite[Section~6]{vanFraassen1974hiddenvariables}, \cite[p.~105]{Veltman1985logics}, \cite[fn.~2]{fine2012difficulty}, \cite{kocurek2025stalnaker}, and \cite{dorr2024logic}.} 
Yet this leaves open the problem of determining the logic of Stalnaker's semantics when first-order quantifiers are added.\footnote{A comparison with modal logic may be instructive here. The modal logic \textsf{S4M} with quantifiers and the Barcan formula is incomplete (see \cite{Kripke1967Lemmon2} and \cite[pp.~265--270]{HughesCresswell1996}). Nevertheless, its semantics (i.e., the class of frames defined by this logic) can be completely axiomatized by extending the logic with a single additional schema (see \cite[pp.~159--164]{Cresswell2001howtocomplete}). Hence, the results of \cite{KocurekWalshWeissA} do not in themselves preclude a solution to the problem of axiomatizing Stalnaker's semantics.} 
Call this \emph{Stalnaker's logical problem of conditionals}. 

In this paper, we prove that Stalnaker's logical problem of conditionals is unsolvable. 
That is, the logic of the class of first-order selection function frames by which Stalnaker characterized the conditional function is not recursively axiomatizable. 
The proof borrows heavily from similar nonaxiomatizability results from first-order modal logic, especially \cite{Cresswell1997incompletable}.\footnote{For related results, see \cite{Kamp1977relatedtheorems}, \cite{Montagna1984provability}, \cite[Section~3]{Garson1984quantification}, and \cite{Rybakov2024incompleteness}.}  
In particular, we demonstrate nonaxiomatizability by showing how to effectively interpret the theory of arithmetic within the logic of Stalnaker's semantics. 
Importantly, the proof does not rely on various design choices for how to integrate quantifiers into the semantics (e.g., variable vs.\ constant domains, whether to add an existence or identity predicate, and so on). 
This result thus answers the open questions that \cite{KocurekWalshWeissA} raise at the end of their paper. 

The plan of our paper is as follows. 
Section~\ref{sec:sem} gives a brief overview of the selection function semantics enriched with first-order quantification. 
In Section~\ref{sec:proof}, we prove that Stalnaker's semantics is not axiomatizable and also highlight some corollaries of this result. 
Section~\ref{sec:conc} concludes with directions for future research and some reflections on what this result means for Stalnakerians and for the study of conditional logic more generally.

\section{Selection Function Semantics}
\label{sec:sem}

Before presenting our result, we'll start by reviewing the selection function semantics for conditionals. 
Readers familiar with this semantics may simply want to review the remarks about the Weak Limit Assumption (starting after the proof of Lemma~\ref{lem:rat}) before moving on to Section~\ref{sec:proof}. 

Throughout, we assume that we have a countable stock of individual variables $(x_i)_{i<\omega}$ and, for each $n$, countably many $n$-place predicates $(P_i^n)_{i<\omega}$. 
The language $\mathcal{L}$ of quantified conditional logic is given below:
    \begin{align*}
        \phi & \Coloneqq P^n (x_1,\dots,x_n) \mid \neg\phi \mid (\phi \then \phi) \mid (\phi \would \phi) \mid \forall x\phi
    \end{align*}
where $P^n$ is an $n$-place predicate. 
The other extensional operators---$\wedge$, $\vee$, $\iff$, $\top$, $\bot$, and $\exists$---are defined in the standard way. 
We also adopt the following standard abbreviations:
\begin{align*}
    \nec\phi & \coloneqq (\neg\phi \would \bot) \\
    \pos\phi & \coloneqq \neg(\phi \would \bot)
\end{align*}
We abuse notation in standard ways and often write things like $x$ and $y$ for individual variables and $P$ and $F$ for predicates, allowing arity to be determined by the context. 


\begin{definition}
\label{Definition:Set Selection Frame}
    A \emph{set selection function frame} (hereafter, \emph{selection frame}) is a quintuple $\mathcal{F} = \tup{W,R,f,D,d}$, where:
    \begin{itemize}
        \item $W \neq \varnothing$ is a set of worlds;
        \item $R \subseteq W \times W$ is the accessibility relation; 
        \item $f:\mathcal{P}(W)\times W \to \mathcal{P}(W)$ is the (set) selection function;
        \item $D \neq \varnothing$ is the (global) domain;
        \item $d\colon W \rightarrow \mathcal{P}(D)$ is the local domain assignment.\footnote{We follow \cite[p.~25]{StalnakerThomason1970} in allowing for empty local domains.}
    \end{itemize}
    We write $R(w)$ for $\{x \in W \mid wRx\}$. The only further condition we impose is that for any $P \subseteq W$ and any $w\in W$, $f(P,w) \subseteq R(w)$.
\end{definition}

Note that we are working in a variable domain framework for the sake of generality. 
However, none of the results in this paper hinge on this: everything carries over even in a global domain setting. 

\begin{definition}
    \label{Definition:SelectionModel}
    A \emph{selection model} is a sextuple $\mathcal{M}=\tup{W,R,f,D,d,I}$, where $\tup{W,R,f,D,d}$ is a selection frame (Definition~\ref{Definition:Set Selection Frame}) and for every $w\in W$ and $n$-ary predicate $P^n$, $I(P^n,w)\subseteq D^n$. 
    We say $\mathcal{M}$ is \emph{based on} the frame $\tup{W,R,f,D,d}$. 
\end{definition}

\begin{definition}
    \label{Definition:VariableAssignment}
    A \emph{variable assignment} in a selection model $\tup{W,R,f,D,d,I}$ is a function $g:\mathcal{V}\to{D}$, where $\mathcal{V}$ is the set of variables. 
    A variable assignment $g'$ is an \emph{$x$-variant} of $g$ (in symbols, $g'\sim_{x}g$) if $g(y)=g'(y)$ for all variables $y$ except possibly $x$. 
    For $a\in D$, we write $g^x_a$ for the $x$-variant of $g$ such that $g^x_a(x) = a$. 
\end{definition}

\begin{definition}
\label{Definition:Satisfaction}
    For a given selection model $\mathcal{M} = \tup{W,R,f,D,d,I}$, world $w\in W$, and variable assignment $g$, we define $\Vdash$ as follows (we write $[\phi]^{g}$ for $\{w\in W \mid \mathcal{M},w,g \Vdash \phi\}$):    
    \begin{itemize}
        \item[At.] $\mathcal{M},w,g \Vdash P^{n}(x_1,\dots,x_n)$ iff $\tup{g(x_1),\dots,g(x_n)}\in I(P^{n},w)$;
        \item[$\neg$.] $\mathcal{M},w,g \Vdash \neg\phi$ iff $\mathcal{M},w,g\not\Vdash \phi$;
        \item[$\then$.] $\mathcal{M},w,g \Vdash \phi\then\psi$ iff $\mathcal{M},w,g\not\Vdash \phi$ or $\mathcal{M},w,g \Vdash \psi$;
        \item[$\forall$.] $\mathcal{M},w,g \Vdash \forall{x}\phi$ iff for all $a \in d(w)$, $\mathcal{M},w,g^x_a \Vdash \phi$;
        \item[$\would$.] $\mathcal{M},w,g \Vdash \phi \would \psi$ iff $f([\phi]^{g},w)\subseteq[\psi]^{g}$.
    \end{itemize}
    Where $\Gamma$ is a set of formulas, we write $\mathcal{M},w,g \Vdash \Gamma$ to mean that $\mathcal{M},w,g \Vdash \gamma$ for every $\gamma\in\Gamma$. 
    We write $\mathcal{M},w \Vdash \phi$ to mean that $\mathcal{M},w,g \Vdash \phi$ for every variable assignment $g$. 
\end{definition}


\begin{definition}
\label{Definition:SelectionValidity}
    We say that $\Gamma$ is \emph{valid in the selection frame} $\mathcal{F} = \tup{W,R,f,D,d}$ (in symbols, $\mathcal{F} \Vdash \Gamma$) iff $\mathcal{M},w,g \Vdash \Gamma$ for every model $\mathcal{M} = \tup{\mathcal{F},I}$, every $w \in W$, and every variable assignment $g$ in $\mathcal{M}$. 
    We say that $\Gamma$ is \emph{valid in the class of selection frames} $\mathcal{C}$ (in symbols, $\mathcal{C} \Vdash \Gamma$) iff for every $\mathcal{F}\in\mathcal{C}$, $\mathcal{F} \Vdash \Gamma$. 

    Given a class of frames $\mathcal{C}$, the \emph{logic of $\mathcal{C}$} (in symbols, $\mathsf{L}(\mathcal{C})$) is the set of formulas that are valid in $\mathcal{C}$. 
    Regarding a logic $\mathsf{L}$ as a set of formulas, we say $\mathsf{L}$ is \emph{frame complete} iff there is a class of frames $\mathcal{C}$ where $\mathsf{L} = \mathsf{L}(\mathcal{C})$. 
    We say $\mathsf{L}$ is \emph{frame incomplete} if it is not frame complete, i.e., there is no class of frames $\mathcal{C}$ where $\mathsf{L} = \mathsf{L}(\mathcal{C})$. 
\end{definition}

Having laid out the selection semantics in generality, we now turn to the specific class of frames that Stalnaker uses to analyze conditionals. 

\begin{definition}
\label{Definition:Stalnakerian Function}
    A selection frame $\tup{ W,R,f,D,d}$ is \emph{Stalnakerian} just in case it satisfies these conditions for all $P,Q\subseteq W$, $w\in W$:
    \begin{enumerate}
        \item $f(P,w) \subseteq P$ \hfill (Success)
        \item If $w \in P$, then $w \in f(P,w)$ \hfill (Weak Centering)
        \item If $f(P,w) = \varnothing$, then $P \cap R(w) = \varnothing$ \hfill (LA)
        \item If $f(P,w) \subseteq Q$ and $f(Q,w) \subseteq P$, then $f(P,w) = f(Q,w)$ \hfill (Uniformity)
        \item $|f(P,w)| \leq 1$ \hfill (Uniqueness)
    \end{enumerate}
    A selection frame is \emph{weakly Stalnakerian} if it satisfies conditions 1, 2, 4, and 5. 
\end{definition}


\begin{remark}
\label{Remark:Elementary Stalnakerian Properties}
    Weak Centering implies that the accessibility relation in any (weakly) Stalnakerian frame is reflexive. 
    Moreover, in the presence of Uniqueness, Weak Centering is equivalent to ``Strong Centering'' ($w \in P$ implies $f(P,w) = \{w\}$). 
    Since we generally assume Uniqueness in what follows, we may refer to either condition as the ``Centering'' condition. 
\end{remark}

In Section~\ref{sec:proof}, we will frequently invoke the fact that (weakly) Stalnakerian selection functions obey a sort of ``rational monotonicity'' constraint:

\begin{lemma}\label{lem:rat}
    Let $\mathcal{F} = \tup{W,R,f,D,d}$ be weakly Stalnakerian. 
    If $Q \subseteq P$ and $f(P,w) \cap Q \neq \varnothing$, then $f(Q,w) = f(P,w)$. 
\end{lemma}

\begin{proof}
    Suppose $f(P,w) \cap Q \neq \varnothing$ and $Q \subseteq P$. 
    By Uniqueness, there is a $v$ such that $f(P,w) = \{v\}$. 
    Thus, $v \in Q$, meaning $f(P,w) \subseteq Q$. 
    By Success, $f(Q,w) \subseteq Q \subseteq P$. 
    Hence, by Uniformity, $f(P,w) = f(Q,w)$. 
\end{proof}

Note that Uniformity and Uniqueness already imply a weak version of the \emph{Limit Assumption} (LA), which we dub the \emph{Weak Limit Assumption} (WLA):
    \begin{enumerate}
        \item[3$^-$.] If $f(P,w) = \varnothing$, then $P \cap f(Q,w) = \varnothing$ \hfill (WLA)
    \end{enumerate}

\begin{remark}
    \label{Remark:StalnakerianvsWeaklyStalnakerian}
    Every Stalnakerian selection frame is weakly Stalnakerian, but not conversely (consult \cite[Section~2]{KocurekWalshWeissA}).
\end{remark}




However, as we'll now show, every weakly Stalnakerian frame is equivalent to some Stalnakerian frame. 

\begin{definition}
    Where $\mathcal{F} = \tup{W,R,f,D,d}$ is a selection frame, the \emph{selection-accessibility} relation $R_f$ is defined as follows: $R_f(w) = \bigcup_{P \subseteq W}f(P,w)$. 
\end{definition}

\begin{remark}
    If $\mathcal{F}$ is Stalnakerian, then $R_f = R$. 
    For by Definition~\ref{Definition:Set Selection Frame}, $f(P,w) \subseteq R(w)$ for all $P \subseteq W$, and thus $R_f(w) \subseteq R(w)$. 
    And by LA, if $v \in R(w)$, then $f(\{v\},w) \neq \varnothing$. 
    By Success, $f(\{v\},w) = \{v\}$, and so $v \in R_f(w)$. 
    Hence, $R(w) \subseteq R_f(w)$. 
\end{remark}

\begin{remark}
    In weakly Stalnakerian models, $\nec$ and $\pos$ have the following derived truth conditions:
    \footnote{
        \emph{Proof} ($\nec$-case): 
        Suppose $\mathcal{M},w,g \Vdash \nec\phi$ (i.e., $\neg\phi \would \bot$). 
        Thus, $f([\neg\phi]^g,w) = \varnothing$. 
        By WLA, $[\neg\phi]^g \cap f(P,w) = \varnothing$ for all $P$. 
        Hence, $[\neg\phi]^g \cap R_f(w) = \varnothing$, i.e., $R_f(w) \subseteq [\phi]^g$. 
        Conversely, suppose $\mathcal{M},w,g \nVdash \nec\phi$. 
        Thus, $f([\neg\phi]^g,w) \neq \varnothing$. 
        Let $v \in f([\neg\phi]^g,w)$. 
        By Success, $v \in R_f(w) \cap [\neg\phi]^g$, i.e., $R_f(w) \nsubseteq [\phi]^g$. 
        } 
    \begin{itemize}
        \item[$\nec$.] $\mathcal{M},w,g \Vdash \nec\phi$ iff $R_f(w)\subseteq[\phi]^{g}$;
        \item[$\pos$.] $\mathcal{M},w,g \Vdash \pos\phi$ iff $R_f(w)\cap[\phi]^{g}\neq\varnothing$.
    \end{itemize}
    In other words, $\nec$ and $\pos$ behave like normal modal operators with respect to $R_f$ (which, for weakly Stalnakerian frames, is at least reflexive). 
\end{remark}

\begin{lemma}\label{lem:convert}
    If $\mathcal{F} = \tup{W,R,f,D,d}$ is a weakly Stalnakerian frame, then $\mathcal{F}^* = \tup{W,R_f,f,D,d}$ is Stalnakerian. 
    Moreover, for all $\phi$, all $w$, all $g$, and all interpretations $I$ over $\mathcal{F}$, where $\mathcal{M} = \tup{\mathcal{F},I}$ and $\mathcal{M}^* = \tup{\mathcal{F}^*,I}$:
    \begin{align*}
        \mathcal{M},w,g \Vdash \phi & \quad\Leftrightarrow\quad \mathcal{M}^*,w,g \Vdash \phi.
    \end{align*}
\end{lemma}

\begin{proof}
    By definition of $R_f$, $f(P,w) \subseteq R_f(w)$. 
    So $\mathcal{F}^*$ is a selection frame. 
    Moreover, if $f(P,w) = \varnothing$, then by WLA, $f(Q,w) \cap P = \varnothing$ for all $Q$, and so $\bigcup_{Q \subseteq W}f(Q,w) \cap P = R_f(w) \cap P = \varnothing$. 
    Hence, $\mathcal{F}^*$ is Stalnakerian. 
    Lastly, the proof that $\mathcal{M}$ and $\mathcal{M}^*$ satisfy the same formulas is by a simple induction on $\phi$.
    In particular, by IH, if $[\phi]^{\mathcal{M},g} = [\phi]^{\mathcal{M}^*,g}$ and $[\psi]^{\mathcal{M},g} = [\psi]^{\mathcal{M}^*,g}$, then $f([\phi]^{\mathcal{M},g},w) \subseteq [\psi]^{\mathcal{M},g}$ iff $f([\phi]^{\mathcal{M}^*,g},w) \subseteq [\psi]^{\mathcal{M}^*,g}$, and thus $\mathcal{M},w,g \Vdash \phi \would \psi$ iff $\mathcal{M}^*,w,g \Vdash \phi \would \psi$. 
\end{proof}

What this shows is that LA does not play a crucial role in the logic of Stalnakerian frames. 
Specifically, let $\WS$ be the class of weakly Stalnakerian frames and $\S$ be the class of Stalnakerian frames. 
Lemma~\ref{lem:convert} immediately yields the following:

\begin{corollary}\label{cor:equiv}
    $\mathsf{L}(\mathcal{WS}) = \mathsf{L}(\mathcal{S})$.
\end{corollary}

\noindent This means that LA, by itself, cannot be blamed for the nonaxiomatizability result proven in Section~\ref{sec:proof}. 
Even in weakly Stalnakerian frames that do not satisfy LA, we are able to define an underlying structure for which LA holds. 
It is the presence of this underlying structure that wreaks havoc on axiomatizability.

\section{Nonaxiomatizability}\label{sec:proof}

We now establish our main result: the logic of (weakly) Stalnakerian frames is not recursively axiomatizable. 
Our proof strategy is heavily inspired by the proof of the nonaxiomatizability of certain quantified modal logics found in \cite{Cresswell1997incompletable}. 
Specifically, our strategy will be to show how to encode the theory of arithmetic into the logic of Stalnakerian frames. 

First, we start with a familiar observation: from a Stalnakerian selection function, we can define a well-ordering over accessible worlds simply by taking $v \preceq_w u$ to be defined as $v \in f(\{v,u\},w)$. 
This point generalizes to weakly Stalnakerian frames, except the well-ordering is over \emph{selection}-accessible worlds. 

\begin{definition}
    Let $\mathcal{F} = \tup{W,R,f,D,d}$ be weakly Stalnakerian. 
    Where $w \in W$ and $v,u \in R_f(w)$, we write:
    \begin{align*}
        v \preceq_w u & \quad\Leftrightarrow\quad v \in f(\{v,u\},w) \\
        v \prec_w u & \quad\Leftrightarrow\quad v \preceq_w u \text{ and } v \neq u.
    \end{align*}
\end{definition}

\begin{lemma}\label{lem:world-order}
    If $\mathcal{F}$ is weakly Stalnakerian, then $\preceq_w$ is a well-ordering of $R_f(w)$. 
\end{lemma}

Lemma~\ref{lem:world-order} already follows from Lemma~\ref{lem:convert} and the well-known fact that $\prec_w$ is a well-order for Stalnakerian frames. 
But we will establish Lemma~\ref{lem:world-order} directly for the sake of completeness. 


\begin{proof}
    First, $\preceq_w$ is reflexive: 
    If $v \in R_f(w)$, then $v \in f(P,w)$ for some $P$. 
    Thus, $f(P,w) \cap \{v\} \neq \varnothing$. 
    By Success, $v \in P$, i.e., $\{v\} \subseteq P$. 
    By Lemma~\ref{lem:rat}, $f(P,w) = f(\{v\},w)$. 
    So $v \in f(\{v\},w)$, and thus $v \preceq_w v$. 

    Next, $\preceq_w$ is transitive: 
    Let $x \preceq_w y \preceq_w z$. 
    Thus, $x \in f(\{x,y\},w)$ and $y \in f(\{y,z\},w)$. 
    If $f(\{x,y,z\},w) = \varnothing$, then by WLA, $\{x,y,z\} \cap f(\{x,y\},w) = \varnothing$, contrary to the fact that $x \in f(\{x,y\},w)$. 
    Hence, $f(\{x,y,z\},w) \neq \varnothing$. 
    Thus, if $x \in f(\{x,y,z\},w)$, then $f(\{x,y,z\},w) \cap \{x,z\} \neq \varnothing$, and so by Lemma~\ref{lem:rat}, $x \in f(\{x,y,z\},w) = f(\{x,z\},w)$, i.e., $x \preceq_w z$. 
    So it suffices to show that $x \in f(\{x,y,z\},w)$. 
    This follows from two facts, which we will prove:
    \begin{enumerate}[label=(\roman*)]
        \item If $z \in f(\{x,y,z\},w)$, then $y \in f(\{x,y,z\},w)$. For if $z \in f(\{x,y,z\},w)$, then $f(\{x,y,z\},w) \cap \{y,z\} \neq \varnothing$. So by Lemma~\ref{lem:rat}, $f(\{x,y,z\},w) = f(\{y,z\},w) \ni y$, and so $y \in f(\{x,y,z\},w)$. 
        \item If $y \in f(\{x,y,z\},w)$, then $x \in f(\{x,y,z\},w)$. For if $y \in f(\{x,y,z\},w)$, then $f(\{x,y,z\},w) \cap \{x,y\} \neq \varnothing$. So by Lemma~\ref{lem:rat}, $f(\{x,y,z\},w) = f(\{x,y\},w) \ni x$, and so $x \in f(\{x,y,z\},w)$. 
    \end{enumerate}
    Thus, since $f(\{x,y,z\},w) \neq \varnothing$, it follows by (i) and (ii) that $x \in f(\{x,y,z\},w)$. 
    
    Next, $\preceq_w$ is antisymmetric: 
    Suppose $x \preceq_w y$ and $y \preceq_w x$. 
    Thus, $x,y \in f(\{x,y\},w)$. 
    By Uniqueness, $x=y$. 
    
    Next, $\preceq_w$ is total: 
    Let $x,y \in R_f(w)$. 
    So for some $P$ and $Q$, $x \in f(P,w)$ and $y \in f(Q,w)$. 
    If $f(\{x,y\},w) = \varnothing$, then by WLA, $\{x,y\} \cap f(P,w) = \varnothing$, contrary to supposition. 
    Hence, $f(\{x,y\},w) \neq \varnothing$. 
    By Success, either $x \in f(\{x,y\},w)$ or $y \in f(\{x,y\},w)$, i.e., either $x \preceq_w y$ or $y \preceq_w x$. 

    Finally, $\preceq_w$ is well-founded: 
    Let $\varnothing \neq P \subseteq R_f(w)$. 
    If $f(P,w) = \varnothing$, then by WLA, $P \cap R_f(w) = \varnothing$, contrary to supposition. 
    So $f(P,w) \neq \varnothing$. 
    By Success and Uniqueness, $f(P,w) = \{x\}$ for some $x \in P$. 
    Let $y \in P$. 
    Since $f(P,w) \cap \{x,y\} \neq \varnothing$, it follows by Lemma~\ref{lem:rat} that $x \in f(P,w) = f(\{x,y\},w)$. 
    Hence, $x \preceq_w y$. 
\end{proof}


Fix a monadic predicate $F$. 
Define the following:
\begin{align*}
    N(x) & \quad \coloneqq \quad \pos F(x) \\
    x < y & \quad \coloneqq \quad N(x) \wedge N(y) \wedge ((F(x) \vee F(y)) \would \neg F(y)) \\
    x \equiv y & \quad \coloneqq \quad N(x) \wedge N(y) \wedge ((F(x) \vee F(y)) \would (F(x) \wedge F(y))) 
\end{align*}
Effectively, our strategy will be to show that in any weakly Stalnakerian model satisfying some formulas $\phi_1,\dots,\phi_n$ characterizing the behavior of these predicates, the objects satisfying $N(x)$ form a structure that is isomorphic to the natural numbers. 
This will allow us to recursively embed the theory of arithmetic within the logic of (weakly) Stalnakerian frames. 

Throughout, where $\mathcal{M}$ is a weakly Stalnakerian model and where $\phi(\vec{x})$ is a formula, we write $\phi_{\mathcal{M},w}$ for $\{\vec{a} \in d(w) \mid \mathcal{M},w \Vdash \phi(\vec{a})\}$. 
We will often leave $\mathcal{M}$ implicit, writing $\phi_w$ when $\mathcal{M}$ is understood. 
Thus, $N_w = \{a \in d(w) \mid \mathcal{M},w \Vdash N(a)\}$, $<_w = \{\tup{a,b} \in d(w)^2 \mid \mathcal{M},w \Vdash a < b\}$, and $\equiv_w = \{\tup{a,b} \in d(w)^2 \mid \mathcal{M},w \Vdash a \equiv b\}$. 
Similar conventions will be employed throughout.

Where $a \in d(w)$, we write $\bar{a}$ for the unique member of $f([F(a)],w)$, assuming there is one; otherwise, $\bar{a}$ is undefined. 
(Technically, we should write $\bar{a}_w$, but we'll leave $w$ implicit for readability.) 
By Uniqueness, $\bar{a}$ is defined (for $a \in d(w)$) iff $f([F(a)],w) \neq \varnothing$. 
Hence:

\begin{lemma}\label{lem:Nw}
    For any $a \in d(w)$, $\bar{a}$ is defined iff $a \in N_w$. 
\end{lemma}

\begin{proof}
Given the definition of $\pos$, $a \in N_w$ iff $\mathcal{M},w \Vdash \neg(F(a) \would \bot)$, which holds iff $f([F(a)],w) \neq \varnothing$, i.e., $\bar{a}$ is defined. 
\end{proof}


\begin{lemma}\label{lem:order}
For $a,b \in N_w$:
    \begin{enumerate}
        \item $a <_w b$ iff $\bar{a} \prec_w \bar{b}$
        \item $a \equiv_w b$ iff $\bar{a} = \bar{b}$
    \end{enumerate}
\end{lemma}

\begin{proof}
By Lemma~\ref{lem:Nw}, since $a,b \in N_w$, $\bar{a}$ and $\bar{b}$ are well-defined throughout. 
\begin{enumerate}
    \item Suppose $a <_w b$. 
    Thus, $f([F(a) \vee F(b)],w) \cap [F(b)] = \varnothing$. 
    Hence, $f([F(a) \vee F(b)],w) \subseteq [F(a)]$. 
    If $f([F(a) \vee F(b)],w) = \varnothing$, then by WLA, $f([F(a)],w) \cap [F(a) \vee F(b)] = f([F(a)],w) = \varnothing$, contrary to supposition. 
    So $f([F(a) \vee F(b)],w) \cap [F(a)] \neq \varnothing$. 
    By Lemma~\ref{lem:rat}, $f([F(a) \vee F(b)],w) = f([F(a)],w) = \{\bar{a}\}$. 
    Since $f([F(a) \vee F(b)],w) \cap \{\bar{a},\bar{b}\} \neq \varnothing$, it follows again by Lemma~\ref{lem:rat} that $f(\{\bar{a},\bar{b}\},w) = \{\bar{a}\}$, i.e., $\bar{a} \preceq_w \bar{b}$. 
    And since $\bar{b} \in [F(b)]$, $\bar{b} \notin f([F(a) \vee F(b)],w) = \{\bar{a}\}$. 
    So $\bar{a} \neq \bar{b}$. Hence, $\bar{a} \prec_w \bar{b}$.

    Conversely, suppose $\bar{a} \prec_w \bar{b}$. 
    Suppose for reductio that $f([F(a) \vee F(b)],w) \cap [F(b)] \neq \varnothing$. 
    By Lemma~\ref{lem:rat}, we'd have $f([F(a) \vee F(b)],w) = f([F(b)],w) = \{\bar{b}\}$. 
    But then $f([F(a) \vee F(b)],w) \cap \{\bar{a},\bar{b}\} \neq \varnothing$, so again by Lemma~\ref{lem:rat}, we'd have $\{\bar{a}\} = f(\{\bar{a},\bar{b}\},w) = f([F(a) \vee F(b)],w) = \{\bar{b}\}$, contrary to the fact that $\bar{a} \neq \bar{b}$. 
    Hence, $f([F(a) \vee F(b)],w) \cap [F(b)] = \varnothing$, i.e., $a <_w b$. 

    \item Suppose $a \equiv_w b$. 
    Thus, $f([F(a) \vee F(b)],w) \subseteq [F(a)] \cap [F(b)]$. 
    Since $f([F(a)],w) = \{\bar{a}\} \subseteq [F(a) \vee F(b)] \supseteq \{\bar{b}\} = f([F(b)],w)$, it follows by Uniformity (twice) that $\{\bar{a}\} = \{\bar{b}\}$, i.e., $\bar{a} = \bar{b}$. 

    Conversely, suppose $\bar{a} = \bar{b}$, i.e., $f([F(a)],w) = f([F(b)],w)$. 
    By Success, $f([F(a) \vee F(b)],w) \subseteq [F(a)] \cup [F(b)]$. 
    And by WLA, $f([F(a) \vee F(b)],w) \neq \varnothing$. 
    So either $f([F(a) \vee F(b)],w) \cap [F(a)] \neq \varnothing$ or $f([F(a) \vee F(b)],w) \cap [F(b)] \neq \varnothing$. 
    Either way, by Lemma~\ref{lem:rat}, $f([F(a) \vee F(b)],w) = f([F(a)],w) = f([F(b)],w)$. 
    By Success, $f([F(a) \vee F(b)],w) \subseteq [F(a)] \cap [F(b)]$, i.e., $a \equiv_w b$. 
\end{enumerate}    
\end{proof}

\begin{corollary}\label{cor:preorder}
    $<_w$ is a strict partial order on $N_w$ and $\equiv_w$ is an equivalence relation on $N_w$. Moreover, for all $a,b \in N_w$, if $a \nless_w b$ and $b \nless_w a$, then $a \equiv_w b$. 
\end{corollary}

We'll write $a_\equiv$ for the $\equiv_w$-equivalence class containing $a$. 
Technically we should write $a_{\equiv_w}$, but we'll leave $w$ implicit. 
Similarly, we will write $N_w/\equiv$ in place of $N_w/\equiv_w$. 

We'll also define $<_w^*$ to be $\{\tup{a_\equiv,b_\equiv} \mid a <_w b\}$. 
That is, $<_w^*$ is the induced order on the equivalence classes.
This is well-defined: if $a' \in a_\equiv$ and $b' \in b_\equiv$, then by Lemma~\ref{lem:order}, $\bar{a} = \bar{a'}$ and $\bar{b} = \bar{b'}$, and so $a' <_w b'$ iff $\bar{a'} \prec_w \bar{b'}$ iff $\bar{a} \prec_w \bar{b}$ iff $a <_w b$. 

\begin{lemma}\label{lem:well-order}
    $<_w^*$ is a well-order on $N_w/\equiv$. 
\end{lemma}

\begin{proof}
    By Corollary~\ref{cor:preorder}, $<_w^*$ is a strict partial order. 
    So we just need to establish that it is total and well-founded. 
    
    First, totality: 
    Suppose $a_\equiv \neq b_\equiv$ and $a_\equiv \nless_w^* b_\equiv$ where $a_\equiv,b_\equiv \in N_w/\equiv$. 
    Thus, $\bar{a} \neq \bar{b}$ and $\mathcal{M},w \nVdash ((F(a) \vee F(b)) \would \neg F(b))$. 
    By Uniqueness, $\mathcal{M},w \Vdash ((F(a) \vee F(b)) \would F(b))$. 
    Since $f([F(a) \vee F(b)],w) \cap [F(b)] \neq \varnothing$, by Lemma~\ref{lem:rat}, we have $f([F(a) \vee F(b)],w) = f([F(b)],w) = \{\bar{b}\}$. 
    Since $f([F(a) \vee F(b)],w) \cap \{\bar{a},\bar{b}\} \neq \varnothing$, by Lemma~\ref{lem:rat} again, we have $f(\{\bar{a},\bar{b}\},w) = f([F(a) \vee F(b)],w) = \{\bar{b}\}$. 
    Hence, $\bar{b} \prec_w \bar{a}$. 
    By Lemma~\ref{lem:order}, $b <_w a$, and thus $b_\equiv <_w^* a_\equiv$. 
    So $<_w^*$ is total. 

    Next, well-foundedness: 
    If $\varnothing \neq X \subseteq N_w/\equiv$, then by Lemma~\ref{lem:world-order}, $\{\bar{c} \mid c_\equiv \in X\}$ (which is well-defined by Lemma~\ref{lem:order}) has a $\prec_w$-minimal element, say, $\bar{d}$. 
    So if there is a $d'_\equiv \in X$ where $d'_\equiv <_w^* d_\equiv$, then $d' <_w d$, which by Lemma~\ref{lem:order} means $\bar{d'} \prec_w \bar{d}$, contrary to the fact that $\bar{d}$ is $\prec_w$-minimal. 
    Hence, $d_\equiv$ is the $<_w^*$-minimal element of $N_w/\equiv$. 
\end{proof}

Now we will show how to encode first-order arithmetic within the logic of weakly Stalnakerian frames. 
The strategy will be to define some formulas that effectively say the objects satisfying $N(x)$ (modded by $\equiv_w$) form a structure that is isomorphic to the natural numbers. 
This is possible since Lemma~\ref{lem:well-order} ensures that the objects satisfying $N(x)$ (modded by $\equiv_w$) form at least a well-order.  

To start, define:
\begin{align*}
    Z(x) & \quad\coloneqq\quad N(x) \wedge \neg\exists y (y < x) \\
    S(x,y) & \quad\coloneqq\quad x < y \wedge \neg\exists z (x < z \wedge z < y)
\end{align*}
We introduce three axioms governing $Z$ and $S$:
\begin{enumerate}[label=(Ax\arabic*),itemsep=1ex]
    \item $\exists x N(x)$\label{ax:zero}
    \item $\forall x (N(x) \then \exists y\ S(x,y))$\label{ax:succ}
    \item $\forall x((N(x) \wedge \neg Z(x)) \then \exists y S(y,x))$\label{ax:pred}
\end{enumerate}

\begin{lemma}\label{lem:nat}
    If $\mathcal{M},w \Vdash \text{\ref{ax:zero}--\ref{ax:pred}}$, then $\tup{N_w/\equiv,<_w^*} \cong \tup{\omega,<}$. 
    Moreover, where $\sigma\colon \tup{N_w/\equiv,<_w^*} \cong \tup{\omega,<}$, $Z_w$ is the $<_w^*$-minimal element of $N_w/\equiv$, i.e., $\sigma(Z_w) = 0$, and $S_w^* = \{\tup{a_\equiv,b_\equiv} \mid S_w(a,b)\}$ is the successor relation over $N_w/\equiv$, i.e., $S_w^*(a_\equiv,b_\equiv)$ iff $\sigma(b_\equiv) = \sigma(a_\equiv)+1$. 
\end{lemma}

\begin{proof}
    By Lemma~\ref{lem:well-order}, $\tup{N_w/\equiv,<_w^*}$ is a well-order. 
    By \ref{ax:zero}, $N_w$ is nonempty. 
    By \ref{ax:succ}, $N_w$ has no $<_w$-maximum element.  
    Together, this means $\tup{N_w/\equiv,<_w^*}$ has a limit ordinal $\lambda \geq \omega$ as its order-type. 
    By \ref{ax:pred}, every non-zero element has an immediate $<_w^*$-predecessor. 
    This means $\tup{N_w/\equiv,<_w^*}$ has an order-type $\lambda \leq \omega$. 
    Hence, $\lambda = \omega$. 
    
    Now, fix an isomorphism $\sigma\colon \tup{N_w/\equiv,<_w^*} \cong \tup{\omega,<}$. 
    First, observe $Z_w \neq \varnothing$ by Lemma~\ref{lem:order} (specifically, $\sigma^{-1}(0) \subseteq Z_w$). 
    Moreover, $Z_w \in N_w/\equiv$: 
    If $a,a' \in Z_w$, then $a \nless_w a'$ and $a' \nless_w a$ by definition of $Z(x)$. 
    Thus, $a_\equiv \nless_w^* a'_\equiv$ and $a'_\equiv \nless_w^* a_\equiv$. 
    By Lemma~\ref{lem:well-order}, $a_\equiv = a'_\equiv$, i.e., $a \equiv_w a'$. 
    Hence, $Z_w \in N_w/\equiv$. 
    By definition of $Z(x)$, there's no $b \in N_w$ such that $b_\equiv <_w^* Z_w$. 
    Hence, $Z_w$ is the $<_w^*$-minimal element, i.e., $\sigma(Z_w) = 0$. 
    
    Finally, $S_w^*$ is the successor relation: 
    $S_w^*(a_\equiv,b_\equiv)$ iff $a <_w b$ and there's no $c \in d(w)$ such that $a <_w c <_w b$. 
    This is true iff $a_\equiv <_w^* b_\equiv$ and there's no $c \in d(w)$ such that $a_\equiv <_w^* c_\equiv <_w^* b_\equiv$. 
    By isomorphism, this is true iff $\sigma(a_\equiv) < \sigma(b_\equiv)$ and there's no $n$ such that $\sigma(a_\equiv) < n < \sigma(b_\equiv)$, i.e., $\sigma(b)$ is the successor of $\sigma(a)$. 
\end{proof}

Fix two 3-place predicates $A$ and $M$. 
Now we introduce the following axioms governing these:
\begin{enumerate}[label=(Ax\arabic*)]
\setcounter{enumi}{3}
    \item $\forall x_1x_2x_3\forall y_1y_2y_3((x_1 \equiv y_1 \wedge x_2 \equiv y_2 \wedge x_3 \equiv y_3) \then ((A(x_1,x_2,x_3) \iff A(y_1,y_2,y_3)) \wedge (M(x_1,x_2,x_3) \iff M(y_1,y_2,y_3))))$\label{ax:cong}
    \item $\forall xyzz'(((A(x,y,z) \wedge A(x,y,z')) \then z \equiv z') \wedge ((M(x,y,z) \wedge M(x,y,z')) \then z \equiv z'))$\label{ax:func}
    \item $\forall xy((N(x) \wedge Z(y)) \then (A(x,y,x) \wedge M(x,y,y)))$\label{ax:unit}
    \item $\forall xyy'zz'((S(y,y') \wedge S(z,z')) \then (A(x,y,z) \then A(x,y',z')))$\label{ax:add}
    \item $\forall xyy'zu((S(y,y') \wedge A(z,x,u)) \then (M(x,y,z) \then M(x,y',u)))$\label{ax:prod}
\end{enumerate}
Let $\Ax \coloneqq \bigwedge \text{\ref{ax:zero}--\ref{ax:prod}}$. 

\begin{lemma}
    Let $\mathcal{M},w \Vdash \Ax$. 
    Let $\sigma\colon \tup{N_w/\equiv,<_w^*} \cong \tup{\omega,<}$. 
    Then for all $a,b,c \in N_w$:
    \begin{enumerate}
        \item $A_w(a,b,c)$ iff $\sigma(a_\equiv) + \sigma(b_\equiv) = \sigma(c_\equiv)$.
        \item $M_w(a,b,c)$ iff $\sigma(a_\equiv) \times \sigma(b_\equiv) = \sigma(c_\equiv)$.
    \end{enumerate}
\end{lemma}

\begin{proof}
    In each case, we proceed by induction on $\sigma(b_\equiv)$.
    \begin{enumerate}
        \item 
        \begin{enumerate}[label=(\roman*)]
            \item \textbf{Base Case:} Suppose $b \in N_w$ where $\sigma(b_\equiv) = 0$. 
            By Lemma~\ref{lem:nat}, $b_\equiv = Z_w$, and so $b \in Z_w$. 
            By \ref{ax:unit}, $A_w(a,b,a)$. 
            \begin{itemize}
                \item[($\Rightarrow$)] Suppose $A_w(a,b,c)$. 
                Then by \ref{ax:func}, $c \equiv_w a$, and so $c_\equiv = a_\equiv$. 
                Hence, $\sigma(a_\equiv) + \sigma(b_\equiv) = \sigma(a_\equiv) + 0 = \sigma(a_\equiv) = \sigma(c_\equiv)$. 
                
                \item[($\Leftarrow$)] Suppose $\sigma(a_\equiv) + \sigma(b_\equiv) = \sigma(c_\equiv)$. 
                Then $\sigma(a_\equiv) + 0 = \sigma(a_\equiv) = \sigma(c_\equiv)$, and so $a \equiv_w c$. 
                Since $A_w(a,b,a)$, it follows by \ref{ax:cong} that $A_w(a,b,c)$. 
            \end{itemize}

            \item \textbf{Inductive Step:} Suppose the claim holds for $b$, where $\sigma(b_\equiv) = n$. 
            Let $\sigma(b'_\equiv) = n+1$. 
            Thus, $b'_\equiv$ is the $S_w^*$-successor of $b_\equiv$, i.e., $S_w(b,b')$. 
            We'll show that the claim holds for $b'_\equiv$. 
            \begin{itemize}
                \item[($\Leftarrow$)] Suppose $\sigma(a_\equiv) + \sigma(b'_\equiv) = \sigma(c'_\equiv)$ for some $a_\equiv$ and $c'_\equiv$. 
                Since $\sigma(b'_\equiv) \neq 0$, $\sigma(c'_\equiv) \neq 0$, and so $c' \notin Z_w$ by Lemma~\ref{lem:nat}. 
                By \ref{ax:pred}, there is a $c \in N_w$ such that $S_w(c,c')$. 
                Thus, by Lemma~\ref{lem:nat}, $\sigma(a_\equiv) + \sigma(b_\equiv) = \sigma(c_\equiv)$. 
                By IH, $A_w(a,b,c)$. 
                By \ref{ax:add}, $A_w(a,b',c')$. 
                
                \item[($\Rightarrow)$] Suppose now that $\sigma(a_\equiv) + \sigma(b'_\equiv) \neq \sigma(c'_\equiv)$ for some $a_\equiv$ and $c'_\equiv$. 
                Since $\sigma$ is an isomorphism, there's a $d' \in N_w$ such that $\sigma(a_\equiv) + \sigma(b'_\equiv) = \sigma(d'_\equiv)$. 
                Thus, by the above reasoning, $A_w(a,b',d')$. 
                Since $\sigma(c'_\equiv) \neq \sigma(d'_\equiv)$, $c' \not\equiv_w d'$. 
                So by \ref{ax:func}, $\neg A_w(a,b',c')$. 
            \end{itemize}
        \end{enumerate}
     
        \item 
        \begin{enumerate}[label=(\roman*)]
            \item \textbf{Base Case:} Suppose $\sigma(b_\equiv) = 0$. 
            By Lemma~\ref{lem:nat}, $b \in Z_w$. 
            By \ref{ax:unit}, $M_w(a,b,b)$. 
            \begin{itemize}
                \item[($\Rightarrow$)] Suppose $M_w(a,b,c)$. 
                Then by \ref{ax:func}, $c \equiv_w b$, and so $\sigma(c_\equiv) = 0$. 
                Hence, $\sigma(a_\equiv) \times \sigma(b_\equiv) = \sigma(a_\equiv) \times 0 = 0 = \sigma(c_\equiv)$. 

                \item[($\Leftarrow$)] Suppose now that $\sigma(a_\equiv) \times \sigma(b_\equiv) = \sigma(c_\equiv)$. 
                Then $\sigma(c_\equiv) = \sigma(a_\equiv) \times 0 = 0$, i.e., $c \equiv_w b$. 
                Since $M_w(a,b,b)$, it follows by \ref{ax:cong} that $M_w(a,b,c)$. 
            \end{itemize}

            \item \textbf{Inductive Step:} Now suppose the claim holds for $b$ where $\sigma(b_\equiv) = n$. 
            Let $\sigma(b'_\equiv) = n+1$. 
            Thus, $b'_\equiv$ is the $S_w^*$-successor of $b_\equiv$. 
            We'll show that the claim holds for $b'_\equiv$. 
            \begin{itemize}
                \item[($\Leftarrow$)] Suppose that $\sigma(a_\equiv) \times \sigma(b'_\equiv) = \sigma(c_\equiv)$ for some $a_\equiv$ and $c_\equiv$. 
                Thus, $\sigma(c_\equiv) \geq \sigma(a_\equiv)$. 
                Since $\sigma$ is an isomorphism, there is a $d$ such that $\sigma(d_\equiv) + \sigma(a_\equiv) = \sigma(c_\equiv)$. 
                (If $\sigma(a_\equiv) = \sigma(c_\equiv)$, we can let $d \in Z_w$.) 
                Thus, $\sigma(a_\equiv) \times \sigma(b_\equiv) = \sigma(d_\equiv)$. 
                By IH, $M_w(a,b,d)$. 
                By \ref{ax:prod}, $M_w(a,b',c)$.

                \item[($\Rightarrow$)] Suppose now that $\sigma(a_\equiv) \times \sigma(b'_\equiv) \neq \sigma(c_\equiv)$ for some $a_\equiv$ and $c_\equiv$. 
                Since $\sigma$ is an isomorphism, there is a $d$ where $\sigma(a_\equiv) \times \sigma(b'_\equiv) = \sigma(d_\equiv)$. 
                Thus, by the above reasoning, $M_w(a,b',d)$. 
                Since $\sigma(c_\equiv) \neq \sigma(d_\equiv)$, $c \not\equiv_w d$. 
                So by \ref{ax:func}, $\neg M_w(a,b',c)$. 
            \end{itemize}
        \end{enumerate}
    \end{enumerate}
\end{proof}

Where $\alpha$ is a formula in the language of arithmetic, let $\alpha^*$ be the standard translation into our abbreviations (restricting quantifiers to $N$). 
Let $\mathcal{N}$ be the standard model of arithmetic. 

\begin{corollary}\label{main-cor}
    Let $\mathcal{M}$ be a weakly Stalnakerian model such that $\mathcal{M},w \Vdash \Ax$. 
    Let $\sigma\colon \tup{N_w/\equiv,<_w^*} \cong \tup{\omega,<}$, let $g$ be a variable assignment on $\omega$, and let $\rho\colon (N_w/\equiv) \rightarrow N_w$ where $\rho(a_\equiv) \in a_\equiv$. 
    Then for any $\alpha$ in the language of arithmetic:
    \begin{align*}
        \mathcal{N},g \vDash \alpha \quad\Leftrightarrow\quad \mathcal{M},w,\rho \circ \sigma^{-1} \circ g \Vdash \alpha^*.
    \end{align*}
\end{corollary}


\begin{theorem}\label{main}
    $\mathcal{N} \vDash \alpha$ iff $(\Ax \then \alpha^*) \in \mathsf{L}(\mathcal{WS})$.
\end{theorem}

\begin{proof}
\begin{itemize}
\item[($\Rightarrow$)] By Corollary \ref{main-cor}.
\item[($\Leftarrow$)] Assume $(\Ax \then \alpha^*) \in \mathsf{L}(\mathcal{WS})$. 
We define the following weakly Stalnakerian model $\mathcal{M}$: 
Let $W = D = \omega$, $d(w)=D$ for all $w\in W$, $R=W^2$, and define:
    \begin{equation*}
        f(X,n) = \begin{cases}
            \{\min_{<}(X - \{m \mid m < n\})\} & \text{if }X - \{m \mid m < n\}\neq\varnothing\\
            \varnothing & \text{otherwise}
        \end{cases}
    \end{equation*}
    Let $I(F,n) = \{n\}$, with $I(A,n)$ and $I(M,n)$ being addition and multiplication. 
    We leave it as an exercise to verify that $\mathcal{M}$ is indeed weakly Stalnakerian and that $\Ax$ is satisfied at the world $0$ in $\mathcal{M}$. 
    Hence, $\alpha^*$ holds at $0$ in $\mathcal{M}$. 
    By Corollary \ref{main-cor}, $\mathcal{N} \vDash \alpha$.
\end{itemize}
\end{proof}

\begin{remark}
Note that the counter-model provided in the $(\Leftarrow)$ direction of the proof of Theorem \ref{main} is \emph{weakly Stalnakerian} but not \emph{Stalnakerian}. One could turn it into a Stalnakerian counter-model by modifying the accessibility relation so that $R(n)=\{k \mid k\geq n\}$. If one wants instead to get a Stalnakerian counter-model with a universal accessibility relation, one need only modify the definition of the selection function as follows:
    \begin{equation*}
        f(X,n) = \begin{cases}
            \{\min_{<}(X - \{m \mid m < n\})\} & \text{if }X - \{m \mid m < n\}\neq\varnothing\\
            \{\min_<(X)\} & \text{if $X - \{m \mid m < n\} = \varnothing$ but $X \neq \varnothing$} \\
            \varnothing & \text{otherwise}
        \end{cases}
    \end{equation*}
\end{remark}

\begin{corollary}\label{cor:wsnonax}
    $\mathsf{L}(\mathcal{WS})$ is not recursively axiomatizable. 
\end{corollary}

\begin{corollary}\label{cor:snonax}
    $\mathsf{L}(\mathcal{S})$ is not recursively axiomatizable. 
\end{corollary}
\begin{proof}
    Immediate from Corollaries~\ref{cor:equiv} and \ref{cor:wsnonax}.
\end{proof}

\begin{remark}
    Indeed, $\mathsf{L}(\mathcal{S})$ (equivalently, $\mathsf{L}(\mathcal{WS})$) is not even \emph{arithmetical} by Tarski's theorem that the first-order theory of $\mathcal{N}$ is not arithmetical. This leaves open the problem of determining the exact complexity of the logic.
\end{remark}

The reader will observe that variable domains played no essential role in the foregoing argument, so the constant domain semantics similarly fails to be recursively axiomatizable. Although we assumed a language with countably many $n$-place predicates (for each $n$), the argument just given can be run in a language with only a single unary predicate and two ternary predicates (all other predicates used in the argument are defined in terms of $F, A$, and $M$). This leaves open the question of which natural fragments of the logic \emph{are} axiomatizable.

\cite{KocurekWalshWeissA} proved that $\mathsf{CQ}$ is frame incomplete. Interestingly, this result can also be extracted (indirectly) as a consequence of Corollary~\ref{cor:wsnonax}.


\begin{corollary}\label{cor:frameincomplete}
    $\mathsf{CQ}$ is frame incomplete.
\end{corollary}
\begin{proof}
    Proposition~1 of \cite{KocurekWalshWeissA} establishes that $\mathsf{CQ}$ is valid on a frame iff that frame is weakly Stalnakerian. 
    Thus, if it were the case that $\mathsf{CQ}=\mathsf{L}(\mathcal{C})$ for some class of selection frames $\mathcal{C}$, then $\mathcal{C}\subseteq\mathcal{WS}$, and so $\mathsf{L}(\mathcal{WS})\subseteq\mathsf{L}(\mathcal{C})=\mathsf{CQ}$. 
    By soundness of $\mathsf{CQ}$, $\mathsf{CQ}\subseteq\mathsf{L}(\mathcal{WS})$, so $\mathsf{CQ}=\mathsf{L}(\mathcal{WS})$. 
    However, $\mathsf{CQ}$ is recursively axiomatizable (since it is formulated axiomatically), whereas $\mathsf{L}(\mathcal{WS})$ is not by Corollary~\ref{cor:wsnonax}, a contradiction.
\end{proof}

\noindent Note, however, that while this argument establishes that $\mathsf{CQ}$ is not the logic of weakly Stalnakerian frames, it does not provide a witness to the incompleteness of \textsf{CQ}, that is, a valid formula unprovable in \textsf{CQ}. 
By contrast, \cite{KocurekWalshWeissA} give a constructive frame incompleteness proof that supplies such a witness. 

Corollaries~\ref{cor:wsnonax}--\ref{cor:snonax} also resolve the problems left open by \cite[Section~6]{KocurekWalshWeissA}. 
For example, they prove that \textsf{CQ} (alias: \textsf{QST}) is frame incomplete by identifying a formula, $\neg\mathsf{DS}$, which is valid in the class of all frames for \textsf{CQ} but not provable in the logic. 
They ask whether the addition of $\neg\mathsf{DS}$ as an axiom to \textsf{CQ} might yield a complete system. 
We can now answer this question in the negative. 
The class of frames on which $\mathsf{CQ}+\neg\mathsf{DS}$ is valid is just $\mathcal{WS}$, and $\mathsf{L}(\mathcal{WS})$ is not recursively axiomatizable. 
Moreover, adding axioms to strengthen the outer modal to $\mathsf{CQ}$ does not change the picture. 
For $\mathsf{CQ}$ is valid on all and only weakly Stalnakerian frames, and there are weakly Stalnakerian models with even universal (selection-)accessibility relations satisfying $\Ax$. 
So strengthening the logic of the outer modal doesn't restore axiomatizability.


\section{Conclusion}
\label{sec:conc}

A solution to the logical problem of conditionals, as conceived by \cite{Stalnaker1968theory}, would be an axiomatization of the logic of a certain kind of conditional function---a function on propositions which obeys the constraints identified by Stalnaker. In this paper, we proved that this problem is unsolvable in a first-order setting. More precisely, the first-order logic of the class of propositional selection function frames identified by Stalnaker is not recursively axiomatizable since it encodes arithmetic.

How one interprets the significance of this result may depend on one's antecedent convictions about the relationship among inference, meaning, and model theory. Those in the tradition of proof-theoretic semantics regard meaning as constituted by inferential role. If one adds to this the natural further thought that the relevant inferential practices must be mechanizable, then a central task of semantics—if not \emph{the} central task of semantics—is the identification of effectively axiomatized systems that codify these inferential practices. Model theory would retain at most an instrumental role in semantics. According to this conception, Stalnaker's semantics overshoots. Indeed, the consequence relation it defines is too complex; it cannot yield the logic of any language whose meanings are fixed by mechanizable rules.

Of course, one might deny that semantic theory is answerable to deduction in this way. One might regard model-theoretic semantics as more directly targeting meaning insofar as it directly models word-world relations. Indeed, some semanticists regard deductive treatments of consequence as ``superfluous'' to semantics \parencite[53]{dowty2012introduction}. Even for those who endorse such a conception of semantics, our result still bears on the theoretical project of characterizing the semantics of the conditional.  In particular, \cite{Holliday2018} have identified “benefits of providing axiomatizations for the purposes of semantic theorizing” that do not depend on an inferential conception of meaning \parencite[73]{Holliday2018}. Chief among the benefits that they highlight is that axiomatizing a model-theoretic semantics renders its entailment predictions perspicuous and empirically assessable. Indeed, they claim that ``we may not fully
understand the predictions of a semantic account until we have an intuitive, complete proof system'' \parencite[77]{Holliday2018}. On the negative side, axiomatization can draw our attention to problematic axioms or inferences; on the positive side, a plausible-looking axiomatization provides assurance that nothing spurious is hiding in the semantics. This assurance is unavailable here, and unavailable in the strongest sense. It is not merely that we have failed to find a complete system; rather, there is no such system to be found.

Those who are moved by the foregoing considerations may want to re-evaluate the principles that yield Stalnaker's non-axiomatizable logic. One \emph{might} have thought that LA was to blame for non-axiomatizability, although we have demonstrated that this is not so. One interesting direction is to study the family of logics that arise when one drops various constraints on the selection function. Note that Uniqueness and Uniformity played an essential role in our proofs. In the presence of these constraints, various choices were equivalent; for example, LA and WLA generated the same logic and WLA itself was derivable. When one drops these constraints, a deluge of logics results. Perhaps some of these logics are both attractive and axiomatizable.

\subsection*{Acknowledgments} 

Our proof search was meaningfully aided by the use of AI (specifically, Anthropic's Fable 5). 
We also used AI to proofread our final draft.
All proofs where AI was used have been independently checked by the authors. 
No AI was used to generate any text in this paper.

\printbibliography

\end{document}